\documentclass[11pt, reqno, oneside, notitlepage]{amsart}

\usepackage[a4paper, total={6in, 9.1in}]{geometry}

\usepackage{amsmath,amscd}
\usepackage{amssymb}
\usepackage{amsthm}
\usepackage{comment}
\usepackage{graphicx}
\usepackage{epstopdf} 
\usepackage{mathrsfs}
\usepackage{cite}

\usepackage{bm}
\usepackage{hyperref}
\usepackage{xcolor}
\hypersetup{
	colorlinks,
	linkcolor={blue},
	urlcolor={blue},
	citecolor={red}
}

\theoremstyle{plain}
\newtheorem{thm}{Theorem}[section]
\newtheorem{prop}{Proposition}[section]

\newtheorem{lem}[prop]{Lemma}

\newtheorem{defi}[prop]{Definition}
\newtheorem{rmk}[prop]{Remark}

\newtheorem*{proposition*}{Proposition}

\numberwithin{equation}{section}

\newcommand {\p} {\partial}

\def\div{\text{div}}

\title[Inverse boundary problem for Mean Field Games]{Inverse boundary problem for a mean field game system with probability density constraint}
\author[H. Liu]{Hongyu Liu}
\address{Department of Mathematics, City University of Hong Kong, Kowloon, Hong Kong SAR, China}
\email{hongyu.liuip@gmail.com, hongyliu@cityu.edu.hk}

\author[S. Zhang]{Shen Zhang}
\address{Department of Mathematics, City University of Hong Kong, Kowloon, Hong Kong SAR, China}
\email{szhang347-c@my.cityu.edu.hk}

\begin{document}
	\maketitle
	
	\begin{abstract}
	By following the study in \cite{MFG2}, we consider an inverse boundary problem for the mean field game system where a probability density constraint is enforced on the game agents. That is, we consider the case that reflective boundary conditions are enforced and hence the population distribution of the game agents should be treated as a probability measure which preserves both positivity and the total population. This poses significant challenges for the corresponding inverse problems in constructing suitable ``probing modes" which should fulfil such a probability density constraint. We develop an effective scheme in tackling such a case which is new to the literature.

		%
	\end{abstract}

	
	\section{Introduction}
	
	\subsection{Problem setup and background}
	
	
	We are concerned with the inverse problems for mean field games (MFGs), which have received significant interest in the literature recently \cite{N1,N2,s9,ILY,s5,Kl23,Kl23-1,KlAv,KLL1,KLL2,s7,LMZ,MFG2,s1,RSK}. We also refer to \cite{model,CarPor,17,HCM06,HCM071,HCM072,HCM073,LL06a,LL06b,LL07a,Lions} for related background on the theory of MFGs. 
	
	By following \cite{MFG2}, we consider the following MFG system in our study: 
	\begin{equation}\label{main} 
		\left\{
		\begin{array}{ll}
			\displaystyle{-\partial_t u(x,t) -\Delta u(t,x)+\frac{1}{2}|\nabla u(x,t)|^2-F(x, m(x,t))=0} &  {\rm{in}}\ Q,\medskip\\
			\displaystyle{\partial_tm(x,t)-\Delta m(x,t)-{\rm div} \big(m(x,t) \nabla u(x,t)\big)=0} & {\rm{in}}\ Q,\medskip\\
			\p_{\nu} u(x,t)=\p_{\nu} m(x,t)=0 & {\rm{on}}\ \Sigma,\medskip\\
			u(x,T)=G,\ m(x,0)=m_0(x) & {\rm{in}}\  \Omega,\medskip
		\end{array}
		\right.
	\end{equation}
where $\Omega\subset\mathbb{R}^n$, $n\in\mathbb{N}$ is a bounded Lipschitz domain, which signifies the state space. Let $\mathcal{P}$ stand for the set of Borel probability measures on $\mathbb{R}^n$, and  $\mathcal{P}(\Omega)$ stand for the set of Borel probability measures on $\Omega$. One has that $m\in\mathcal{P}(\Omega)$ denotes the population distribution of the agents and $u(x, t):\Omega\times [0, T]\mapsto \mathbb{R}$ denotes the value function of each player. Here, $T\in\mathbb{R}_+$ signifies the terminal time in what follows. Moreover, $\Sigma:=\p\Omega\times[0,T]$ , $Q:=\overline{\Omega}\times[0,T]$ and $\nu$ is the exterior unit normal to $\partial\Omega$. In \eqref{main}, $F:\Omega\times\mathcal{P}(\Omega)\mapsto\mathbb{R} $ is the running cost function which signifies the interaction between the agents and the population; $m_0$ is the initial population distribution and $G$ is a constant and it signifies the terminal cost.

	
	When the  players have to control a process in a bounded domain $\Omega\subset\mathbb{R}^n$ with reflexion on the boundary of $\Omega$, it is constrained with the  Neumann boundary conditions. Many economic and financial models are described by this system; see for instance, the models in \cite{model}. Due to this fact, we can define
	\begin{equation}\label{eq:distr1}
		\mathcal{O}:=\{ m:\Omega\to [0,\infty) \ \ |\ \ \int_{\Omega} m\, dx =1 \}.
	\end{equation}
	If $m\in \mathcal{O}$, then it is the density of a distribution in $\Omega$. It can directly verified from \eqref{main} that if the initial distribution $m_0\in\mathcal{O}$, then $m(\cdot; t)\in\mathcal{O}$ for any subsequent time $t$. It is clear that by scaling, the total population $1$ in \eqref{eq:distr1} can be replaced by any positive number. Hence, it is required that $m$ preserves the positivity, namely $m\geq 0$, as well as the number of the total population, namely $\int_\Omega m=1$, for all the time $t\in [0, T]$. This is referred to the \emph{probability density constraint} in our study. That is, the MFG system \eqref{main} is enforced with such a probability density constraint.

We next introduce the measurement/observation dataset for our inverse problem study.
	We define 
	\begin{equation}\label{eq:meop0}
		\mathcal{N}_F(m_0):=\Big(u(x,0),m(x,T)\Big), 
	\end{equation}
	where
	$(u, m)$ is the (unique) pair of solutions to the MFG system \eqref{main} associated with the initial population distribution $m(x, 0)=m_0(x)$.  The inverse problem that we aim to investigate can be formulated as follows:
	\begin{equation}\label{eq:ip1}
		\mathcal{N}_{F}(m_0)\rightarrow F\quad \mbox{for all}\ \ m_0\in\mathcal{H}\subset\mathcal{O}, 
	\end{equation}
	where $\mathcal{H}$ is a proper subset and it shall be described in more details in what follows.

	As discussed earlier, $m$ must fulfil the probability density constraint \eqref{eq:distr1}. This poses significant challenges for the inverse problem study. In this paper, we develop a new strategy that can effectively deal with such a challenging issue.

	The rest of the paper is organized as follows. In Section 2, we fix some notations and introduce several auxiliary results as well as state the main result of the inverse problem. Section 3 is devoted to the study of the forward problem, and Sections~4 and 5 are devoted to the proof of the main result. 
	
	%
	%
	%
	%
	%
	
	\section{Preliminaries and statement of main results}
	
	We adopt the notations in \cite{MFG2}, and mainly work in H\"older spaces $C^{k+\alpha}(\Omega)$, $C^{k+\alpha,\frac{k+\alpha}{2}}(Q) $ for $k\in\mathbb{N}$ and $0<\alpha<1$.  
	
	Next, we introduce the admissibility condition on $F$ , which shall be mainly needed in our study of the inverse problem. 
	
	\begin{defi}\label{Admissible class2}
		We say $U(x,z):\mathbb{R}^n\times\mathbb{C}\to\mathbb{C}$ is admissible, denoted by $U\in\mathcal{A}$, if it satisfies the following conditions:
		\begin{enumerate}
			\item[(i)] The map $z\mapsto U(x,z)$ is holomorphic with the value in $C^{2+\alpha}(\Omega)$ for $0<\alpha<1$.
			\item[(ii)] $U(x,1)=0$ for all $x\in\mathbb{R}^n$. Here we recall that we assume $|\Omega|=1.$
			\item[(iii)]  $U^{(1)}$ is a positive real number. 
		\end{enumerate} 	
		
		Clearly, if (1) and (2) are fulfilled, then $U$ can be expanded into a power series as follows:
		\begin{equation}\label{eq:G}
			U(x,z)=\sum_{k=1}^{\infty} U^{(k)}(x)\frac{(z-1)^k}{k!},
		\end{equation}
		where $ U^{(k)}(x)=\frac{\p^kU}{\p z^k}(1).$
	\end{defi}
For more discussion for this admissible condition, see \cite{MFG2}.

\begin{rmk}
	Note that if $F(x,1)=0$ , then
	$(u,m)=(G,1)$ is a solution of the MFG system \eqref{main}. In this case, the initial distribution $m(x, 0)=1$. This is a common nature of MFG system that the uniform distribution is a stable state.
\end{rmk}

\subsection{Main unique identifiability result}

We are in a position to state the main result.

\begin{thm}\label{der F}
 Assume that $F_j \in\mathcal{A}$ ($j=1,2$) . Let $\mathcal{N}_{F_j}$ be the measurement map associated to
	the following system:
	\begin{equation}\label{eq:mfg1}
		\begin{cases}
			-\p_tu(x,t)-\Delta u(x,t)+\frac 1 2 {|\nabla u(x,t)|^2}= F_j(x,m(x,t)),& \text{ in }  Q,\medskip\\
			\p_t m(x,t)-\Delta m(x,t)-{\rm div} (m(x,t) \nabla u(x,t))=0,&\text{ in } Q,\medskip\\
			\p_{\nu} u(x,t)=\p_{\nu}m(x,t)=0      &\text{ on } \Sigma,\medskip\\
			u(x,T)=G, & \text{ in } \Omega,\medskip\\
			m(x,0)=m_0(x), & \text{ in } \Omega.\\
		\end{cases}  		
	\end{equation}
     If 	
	$$\mathcal{N}_{F_1}(m_0)=\mathcal{N}_{F_2}(m_0),$$  
	for all  $m_0\in  C^{2+\alpha}(\Omega) \cap \mathcal{O}$, where $\mathcal{O}$ is defined in \eqref{eq:distr1},
	  then it holds that 
	$$F_1(x,z)=F_2(x,z)\ \text{  in  } \mathbb{R}.$$ 
\end{thm}


\section{Auxiliary results on the forward problem}\label{section wp}

In this section, we derive several auxiliary results on the forward problem of the MFG system \eqref{main}. One of the key results is the infinite differentiability of the system with respect to small variations around (the density of) a uniform
distribution ($m_0(x)$). We list the results we use here and one may refer to \cite{MFG2}

\begin{lem}\label{linear app unique}
	Assume that   $F^{(1)}\in C^{\alpha}(\Omega)$.
	For  any  $g,\tilde g\in C^{2+\alpha}(\overline\Omega)$,  
	and  $h,\tilde{h}\in C^{\alpha,\alpha/2}(\overline Q)$ with  the   compatibility conditions:
	\begin{align}\label{c-systems}
		\p_{\nu}\tilde g(x)= 
		\p_{\nu} g(x)=0,
	\end{align}
	the following system
	\begin{equation}\label{surjective}
		\begin{cases}
			-u_t-\Delta u-F^{(1)}m=h & \text{ in } Q,\medskip\\
			m_t-\Delta m-\Delta u=\tilde{h}  & \text{ in } Q,\medskip\\
			\p_{\nu}u(x,t)=\p_{\nu}m(x,t)=0     & \text{ on } \Sigma,\medskip\\
		u(x,T)=g    & \text{ in } \Omega,\medskip\\
			m(x,0)=\tilde{g} & \text{ in } \Omega.\\
		\end{cases}
	\end{equation}
	admits a pair of solutions $(u,m)\in [ C^{2+\alpha,1+\alpha/2}(\overline Q)]^2$. 
\end{lem}

Now we present the proof of the local well-posedness of the MFG system \eqref{main}, which shall be needed in our subsequent inverse problem study.

\begin{thm}\label{local_wellpose}
	 Suppose that $F\in\mathcal{A}$ . The following results hold:
	\begin{enumerate}
		
		\item[(a)]
		There exist constants $\delta>0$ and $C>0$ such that for any 
		\[
		m_0\in B_{\delta}(C^{2+\alpha}(\Omega) :=\{m_0\in C^{2+\alpha}(\Omega): \|m_0\|_{C^{2+\alpha}(\Omega)}\leq\delta \},
		\]
		the MFG system $\eqref{main}$ has a solution $(u,m)\in
		[C^{2+\alpha,1+\frac{\alpha}{2}}(Q)]^2$ which satisfies
		\begin{equation}\label{eq:nn1}
			\|(u,m)\|_{ C^{2+\alpha,1+\frac{\alpha}{2}}(Q)}:= \|u\|_{C^{2+\alpha,1+\frac{\alpha}{2}}(Q)}+ \|m\|_{C^{2+\alpha,1+\frac{\alpha}{2}}(Q)}\leq C\|m_0\|_{ C^{2+\alpha}(\Omega)}.
		\end{equation}
		Furthermore, the solution $(u,m)$ is unique within the class
		\begin{equation}\label{eq:nn2}
			\{ (u,m)\in  C^{2+\alpha,1+\frac{\alpha}{2}}(Q)\times C^{2+\alpha,1+\frac{\alpha}{2}}(Q): \|(u,m)\|_{ C^{2+\alpha,1+\frac{\alpha}{2}}(Q)}\leq C\delta \}.
		\end{equation}
		
		\item[(b)] Define a function 
		\[
		S: B_{\delta}(C^{2+\alpha}(\Omega)\to C^{2+\alpha,1+\frac{\alpha}{2}}(Q)\times C^{2+\alpha,1+\frac{\alpha}{2}}(Q)\ \mbox{by $S(m_0):=(u,v)$}, 
		\] 
		where $(u,v)$ is the unique solution to the MFG system \eqref{main}.
		Then for any $m_0\in B_{\delta}(C^{2+\alpha}(\Omega))$, $S$ is holomorphic.
	\end{enumerate}
\end{thm}

\section{ A-priori estimates and analysis of the linearized systems  }\label{analysis of lin}

\subsection{Higher-order linearization}\label{HLM}

We next develop a high-order linearization scheme of the MFG system \eqref{main} in the probability space around a uniform distribution. For detail, we refer to \cite{MFG2}. We just list the  linearization system here.

Let 
$$m_0(x;\varepsilon)=\frac{1}{|\Omega|}+\sum_{l=1}^{N}\varepsilon_lf_l=1+ \sum_{l=1}^{N}\varepsilon_lf_l,$$
where 
\[
f_l\in C^{2+\alpha}(\mathbb{R}^n)\quad\mbox{and}\quad\int_{\Omega} f_l(x) dx =0,
\]
and $\varepsilon=(\varepsilon_1,\varepsilon_2,...,\varepsilon_N)\in\mathbb{R}^N$ with 
$|\varepsilon|=\sum_{l=1}^{N}|\varepsilon_l|$ small enough.

Let
$$u^{(1)}:=\p_{\varepsilon_1}u|_{\varepsilon=0}=\lim\limits_{\varepsilon\to 0}\frac{u(x,t;\varepsilon)-u(x,t;0) }{\varepsilon_1},$$
$$m^{(1)}:=\p_{\varepsilon_1}m|_{\varepsilon=0}=\lim\limits_{\varepsilon\to 0}\frac{m(x,t;\varepsilon)-m(x,t;0) }{\varepsilon_1}.$$

 Now, we have  that $(u_{j}^{(1)},m_{j}^{(1)} )$ satisfies the following system:
\begin{equation}\label{linear l=1,eg}
	\begin{cases}
		-\p_tu^{(1)}(x,t)-\Delta u^{(1)}(x,t)= F^{(1)}m^{(1)}(x,t)& \text{ in } Q,\medskip\\
		\p_t m^{(1)}(x,t)-\Delta m^{(1)}(x,t)-\Delta  u^{(1)}(x,t)=0&\text{ in } Q,\medskip\\
	   u^{(1)}_j(x,T)=0 & \text{ in } \Omega,\medskip\\
		m^{(1)}_j(x,0)=f_1(x). & \text{ in } \Omega.\\
	\end{cases}  
\end{equation}

Then we can define $$u^{(l)}:=\p_{\varepsilon_l}u|_{\varepsilon=0}=\lim\limits_{\varepsilon\to 0}\frac{u(x,t;\varepsilon)-u(x,t;0) }{\varepsilon_l},$$
$$m^{(l)}:=\p_{\varepsilon_l}m|_{\varepsilon=0}=\lim\limits_{\varepsilon\to 0}\frac{m(x,t;\varepsilon)-m(x,t;0) }{\varepsilon_l},$$
for all $l\in\mathbb{N}$ and obtain a sequence of similar systems.

we consider 
\begin{equation}\label{eq:ht1}
	u^{(1,2)}:=\p_{\varepsilon_1}\p_{\varepsilon_2}u|_{\varepsilon=0},
	m^{(1,2)}:=\p_{\varepsilon_1}\p_{\varepsilon_2}m|_{\varepsilon=0}.
\end{equation}

We have the second-order linearization as follows:
\begin{equation}\label{linear l=1,2 eg}
	\begin{cases}
		-\p_tu^{(1,2)}-\Delta u^{(1,2)}(x,t)+\nabla u^{(1)}\cdot \nabla u^{(2)}\medskip\\
		\hspace*{3cm}= F^{(1)}m^{(1,2)}+F^{(2)}(x)m^{(1)}m^{(2)},& \text{ in } \Omega\times(0,T),\medskip\\
		\p_t m^{(1,2)}-\Delta m^{(1,2)}-\Delta u^{(1,2)}= {\rm div} (m^{(1)}\nabla u^{(2)})+{\rm div}(m^{(2)}\nabla u^{(1)}) ,&\text{ in } \Omega\times (0,T),\medskip\\
	    u^{(1,2)}(x,T)=0 & \text{ in } \Omega,\medskip\\
		m^{(1,2)}(x,0)=0, & \text{ in } \Omega.\\
	\end{cases}  	
\end{equation}
Notice that the non-linear terms of the system $\eqref{linear l=1,2 eg}$ depend on the first-order linearised system $\eqref{linear l=1,eg}$. Since we shall make use of the mathematical induction to recover the high-order Taylor coefficients of $F$. This shall be an important ingredient in our proof of Theorem~\ref{der F} in what follows. 

Similarly, for $N\in\mathbb{N}$, we consider 
\begin{equation*}
	u^{(1,2...,N)}=\p_{\varepsilon_1}\p_{\varepsilon_2}...\p_{\varepsilon_N}u|_{\varepsilon=0},
\end{equation*}
\begin{equation*}
	m^{(1,2...,N)}=\p_{\varepsilon_1}\p_{\varepsilon_2}...\p_{\varepsilon_N}m|_{\varepsilon=0}.
\end{equation*}

\subsection{Construction of probing modes}
\begin{thm}\label{con_of_pb}
	Consider the system
	\begin{equation}\label{F is constant}
		\begin{cases}
			-u_t-\Delta u= Fm & \text{ in } Q,\medskip\\
			m_t-\Delta m-\Delta u=0  & \text{ in } Q,\medskip\\
			\p_{\nu}m(x,t)=\p_{\nu}u(x,t)=0  & \text{ on } \Sigma,\medskip\\
			u(x,T)=0& \text{ on } \Omega.\\
		\end{cases}
	\end{equation}
Suppose $F(x)=c$ is constant, then there exist a sequence of $\lambda_i\geq 0$ and $D_i\in\mathbb{R}$such that 
$$(u,m)=(u, -\lambda_ie^{-\lambda_i t}\overline{m}_i(x)+D_i e^{\lambda_i t}\overline{m}_i(x))$$
are solutions of system $\eqref{F is constant}$, where $\overline{m}_i(x)$ are normalized Neumann eigen-functions of of $-\Delta$ in $\Omega$. 
\end{thm}
\begin{proof}
	Let $\beta_i$ be a positive Neumann eigen-value of of $-\Delta$ in $\Omega$ and $\overline{m}_i(x)$ is the corresponding normalized eigen-function.
	In other words, we have $\overline{m}_i(x)$ are not $0$ functions and
	\begin{equation}
		\begin{cases}
			-\Delta \overline{m}_i(x)=\beta_i\overline{m}_i(x)   & \text{ in }\Omega\\
			\p_{\nu}\overline{m}_i(x)=0 & \text{ in }\Sigma. 
		\end{cases}
	\end{equation} 
	
	Now we choose $\lambda_i=\sqrt{\beta_i^2+c\beta_i}$ and $k_i=\beta_i-\lambda_i\leq 0$, then we can check  $$(u,m)=(\frac{c+k_i}{\lambda_i}e^{-\lambda_it}\overline{m}_i(x), e^{-\lambda_i t}\overline{m}_i(x))$$ are solutions of the following system.
		\begin{equation}\label{F is constant'}
		\begin{cases}
			-u_t-\Delta u= F(x)m & \text{ in } Q,\medskip\\
			m_t-\Delta m-\Delta u=0  & \text{ in } Q,\medskip\\
			\p_{\nu}m(x,t)=\p_{\nu}u(x,t)=0  & \text{ on } \Sigma,\medskip\\
		\end{cases}
	\end{equation}
	Notice that
	\begin{equation*}
	\begin{aligned}
		-u_t-\Delta u&=(c+k_i)e^{-\lambda_it}\overline{m}(x)-\frac{c+k_i}{\lambda_i}e^{-\lambda_it}\Delta\overline{m}(x)\\
		&=e^{-\lambda_it}\overline{m}(x)[ c+k_i+\frac{(c+k_i)\beta_i}{\lambda_i}]\\
		&=cm(x,t)+ e^{-\lambda_it}\overline{m}(x)[k_i+\frac{(c+k_i)\beta_i}{\lambda_i}],
	\end{aligned} 
\end{equation*}
and
	\begin{equation*}
		\begin{aligned}
			m_t-\Delta m-\Delta u&=-\lambda_i e^{-\lambda_it}\overline{m}(x)-e^{-\lambda_it}\Delta\overline{m}(x)-\frac{c+k_i}{\lambda_i}e^{-\lambda_it}\Delta\overline{m}(x)\\
			&=-e^{-\lambda_it}\overline{m}(x)[\lambda_i-(\lambda_i+k_i)-\frac{c+k_i}{\lambda_i}\beta_i]\\
			&= e^{-\lambda_it}\overline{m}(x)[k_i+ \frac{(c+k_i)\beta_i}{\lambda_i}  ].             
		\end{aligned}
	\end{equation*}
	Hence, we only need to check $k_i+ \frac{(c+k_i)\beta_i}{\lambda_i}=0$, and it can be shown by the choice of $\lambda_i$ and $k_i$.

	Similarly, we have 
	$$(u,m)=(\frac{c}{\beta_i-\lambda_i}e^{\lambda_it}\overline{m}_i(x), e^{\lambda_i t}\overline{m}_i(x))$$
	are solutions of $\eqref{F is constant'}.$
	
	Then we may choose $D_i=\dfrac{c}{k_i(c+k_i)}\leq 0$ (Since $k\leq0$ and $c+k_i\geq 0$). By the linearity of this system, we have there is a solution of $\eqref{F is constant}$ in the form
	
	$$(u,m)=(u, -\lambda_ie^{-\lambda_i t}\overline{m}_i(x)+D_i e^{\lambda_i t}\overline{m}_i(x)).$$
\end{proof}

\begin{thm}\label{con_of_pb2}
		Consider the system
	\begin{equation}\label{F is constant2}
		\begin{cases}
			u_t-\Delta u= Fm & \text{ in } Q,\medskip\\
			-m_t-\Delta m-\Delta u=0  & \text{ in } Q,\medskip\\
			\p_{\nu}m(x,t)=\p_{\nu}u(x,t)=0  & \text{ on } \Sigma,\medskip\\
		\end{cases}
	\end{equation}
	Suppose $F(x)=c$ is constant, then there exist a sequence of $\lambda_i\geq 0$ and $D_i\in\mathbb{R}$ such that 
	$$(u,m)=(u, -\lambda_ie^{-\lambda_i(T-t)}\overline{m}_i(x)+D_i e^{\lambda_i (T-t)}\overline{m}_i(x))$$
	$$(u,m)=(u,e^{-\lambda_i t}\overline{m}_i(x))$$
	are solutions of system $\eqref{F is constant}$, where $\overline{m}_i(x)$ are normalized Neumann eigen-functions of of $-\Delta$ in $\Omega$. 
\end{thm}
\begin{proof}
	The first pare  is obtained by Theorem $\ref{con_of_pb}$ by a simply changing variable. The second part is true because there is no initial boundary condition in $\eqref{con_of_pb2}.$
\end{proof}

\section{Proof of Theorem ~\ref{der F}}

Before the main proof, we present a key observation as a lemma first. 
\begin{lem}\label{key}
	Let $(v,\rho)$ be a solution of the following system:
	\begin{equation}
		\begin{cases}
			v_t-\Delta v= F_1(x)\rho & \text{ in } Q,\medskip\\
			-\rho_t-\Delta \rho-\Delta v=0  & \text{ in } Q, \medskip\\
		   \p_{\nu}v(x,t)=\p_{\nu}\rho(x,t)=0      & \text{ on } \Sigma, \medskip\\
		\end{cases}
	\end{equation}
	Let $(\overline{u},\overline{m})$ satisfy
	\begin{equation}
		\begin{cases}
			-\overline{u}_t-\Delta \overline{u}- F_1(x)\overline{m}=(F_1-F_2)m_2  & \text{ in } Q, \medskip\\
			\overline{m}_t-\Delta \overline{m}-\Delta \overline{u}=0  & \text{ in } Q, \medskip\\
			\p_{\nu}\overline{u}(x,t)=	\p_{\nu}\overline{m}(x,t)=0    & \text{ on } \Sigma, \medskip\\
			\overline{u}(x,T)=0	   &\text{ in } \Omega, \medskip\\
			\overline{u}(x,0)=0    & \text{ in } \Omega, \medskip\\
			\overline{m}(x,0)=\overline{m}(x,T)=0& \text{ in } \Omega\\
		\end{cases}
	\end{equation}
	Then we have 
	\begin{equation}\label{implies F_1=F_2}
		\int_Q \quad (F_1-F_2)m_2\rho \ dxdt=0.
	\end{equation}
\end{lem}
\begin{proof}
	Notice that
	\begin{equation}\label{IBP1}
		\begin{aligned}
			0&=\int_Q (\overline{m}_t-\Delta \overline{m}-\Delta \overline{u} )\rho\ dxdt\\
			&=\int_{\Omega}\overline{m}\rho\Big|_0^T dx-\int_Q\overline{m}\rho_t dxdt-\int_Q\rho(\Delta\overline{m}+\Delta\overline{u})\ dxdt\\
			&=-\int_Q\overline{m}( -\Delta\rho-\Delta v)dxdt-\int_Q\rho( \Delta\overline{m}+\Delta\overline{u})\ dxdt\\
			&=\int_Q (\overline{m}\Delta v-\rho\Delta\overline{u})\ dxdt.
		\end{aligned}
	\end{equation}
	
	Similarly, one can deduce that
	\begin{equation}\label{IBP2}
		\begin{aligned}
			0&=\int_Q (\overline{m}_t-\Delta \overline{m}-\Delta \overline{u} )v\ dxdt\\
			&=-\int_Q\overline{m}(\Delta v+F_1\rho)dxdt-\int_Q v(\Delta\overline{m}+\Delta\overline{u} )\ dxdt\\
			&=-\int_Q(2\overline{m}\Delta v+F_1\rho\overline{m}+\overline{u} \Delta v)\ dxdt.
		\end{aligned}
	\end{equation}
	Then we have 
	\begin{align*}
		\int_Q (F_1-F_2)m_2\rho dxdt&=\int_Q\quad \rho(-\overline{u}_t-\Delta\overline{u}-F_1\overline{m})\ dxdt\\
		&=\int_{\Omega}\overline{u}\rho\Big|_0^T dx+\int_Q\quad (\rho_t\overline{u}-\rho\Delta\overline{u}-F_1\rho\overline{m})\  dxdt\\
		&=\int_{\Omega}  \overline{u}(x,T)\rho(x,T) dx+ \int_Q\quad (-\Delta\rho-\Delta v)\overline{u}-\rho\Delta\overline{u}-F_1\rho\overline{m}\ dxdt\\
		&= \int_Q\quad (-2\rho\Delta\overline{u}-\overline{u}\Delta v-F_1\rho\overline{m})\ dxdt\\
		&=\int_Q\quad (-2\rho\Delta\overline{u}-\overline{u}\Delta v-F_1\rho\overline{m})\ dxdt, 
	\end{align*}
	which in combination with $\eqref{IBP1}$ and $\eqref{IBP2}$ readily yields that 
	$$	\int_Q \quad (F_1-F_2)m_2\rho \ dxdt=0.$$
	
	The proof is complete. 
\end{proof}


With all the preparations, we are in a position to present the proof of Theorem~\ref{der F}. 

\begin{proof}[ Proof of Theorem $\ref{der F}$ ]
	For $j=1,2$, let us consider 
	\begin{equation}\label{MFG 1,2}
		\begin{cases}
			-u_t-\Delta u+\frac{1}{2}|\nabla  u|^2= F_j(x,m) & \text{ in } Q,\medskip\\
			m_t-\Delta m-\div (m\nabla u)=0  & \text{ in } Q, \medskip\\
			\p_{\nu}u(x,t)=\p_{\nu}m(x,t)=0     & \text{ on } \Sigma, \medskip\\
			u(x,T)=G    & \text{ in } \Omega,\medskip\\
			m(x,0)=m_0(x) & \text{ in } \Omega.\\
		\end{cases}
	\end{equation}
	Next, we divide our proof into three steps. 
	
	\bigskip
	\noindent {\bf Step I.}~First, we do the first order linearization to the MFG system \eqref{MFG 1,2} in $Q$ and can derive: 
	\begin{equation}\label{linearization}
		\begin{cases}
			-\p_{t}u^{(1)}_j-\Delta u_j^{(1)}= F_j^{(1)}(x)m_j^{(1)} & \text{ in } Q, \medskip\\
			\p_{t}m^{(1)} _j-\Delta m_j^{(1)} -\Delta u_j^{(1)}=0  & \text{ in } Q, \medskip\\
			u^{(1)}_j(x,T)=0     & \text{ in } \Omega,\medskip\\
			m^{(1)} _j(x,0)=f_1(x) & \text{ in } \Omega.\\
		\end{cases}
	\end{equation}
	Let $\overline{u}^{(1)}=u^{(1)}_1-u^{(1)}_2$ and $ \overline{m}^{(1)}=m^{(1)} _1-m^{(1)} _2. $ Let 
	$(v,\rho)$ be a solution to the following system
	\begin{equation}\label{adjoint}
		\begin{cases}
			v_t-\Delta v= F^{(1)}_1(x)\rho & \text{ in } Q, \medskip\\
			-\rho_t-\Delta \rho-\Delta v=0  & \text{ in } Q, \medskip\\
			\p_{\nu}v(x,t)=\p_{\nu}\rho(x,t)=0      & \text{ on } \Sigma, \medskip\\
			\end{cases}
	\end{equation}
	Since $\mathcal{N}_{F_1}=\mathcal{N}_{F_2}$, 
	by Lemma $\ref{key}$, we have 
	\begin{equation}\label{implies to 0;1}
		\int_Q \quad( F^{(1)}_1-F^{(1)}_2)m^{(1)} _2\rho \ dxdt =0,
	\end{equation}
	for all $ m^{(1)} _2\in C^{1+\frac{\alpha}{2},2+\alpha}(Q)$ with $m^{(1)} _2 $ being a solution 
	to $\eqref{linearization}.$
	By Theorems $\ref{con_of_pb}$ and $\ref{con_of_pb2}$, we may choose 
	$$m_2^{(1)}(x,t)=-\lambda_ie^{-\lambda_i t}\overline{m}_i(x)+D_i e^{\lambda_i t}\overline{m}_i(x),$$
	$$\rho(x,t)= -\lambda_ie^{-\lambda_i (T-t)}\overline{m}_i(x)+D_i e^{\lambda_i (T-t)}\overline{m}_i(x),$$
	where the notations are the same as the notations in the proof of Theorem $\ref{con_of_pb}.$
	
	In fact, we use the $f_1(x)=-\lambda_i\overline{m}_i(x)+D_i\overline{m}_i(x)$ here. Since $\overline{m}_i(x)$ are Neumann eigen-function of $-\Delta$, the condition $\int_{\Omega}f_1(x)dx=0$ is satisfied. Furthermore, we can determine $m_2^{(1)}(x,t)$ by $f_1(x)$ and equation $\eqref{linearization}$ by Lemma $\ref{linear app unique}.$

	Then $\eqref{implies to 0;1}$ implies that
	\begin{equation}
		(F_1-F_2)\int_{0}^{T}(-\lambda_ie^{-\lambda_i t}+D_i e^{\lambda_i t})(-\lambda_ie^{-\lambda_i (T-t)}+D_i e^{\lambda_i (T-t)})dt\int_{\Omega}\overline{m}_i(x)^2dx=0. 
	\end{equation}
	
	Notice that $\int_{\Omega}\overline{m}_i(x)^2dx=1$ and $(-\lambda_ie^{-\lambda_i t}+D_i e^{\lambda_i t})(-\lambda_ie^{-\lambda_i (T-t)}+D_i e^{\lambda_i (T-t)})$ is positive, we have 
	
	$$F_1=F_2.$$
	\medskip
	
	\noindent{\bf Step II.}~We proceed to consider the second linearization to the MFG system $\eqref{MFG 1,2}$ in $Q$ and can obtain for $j=1,2$:
	\begin{equation}
		\begin{cases}
			-\p_tu_j^{(1,2)}-\Delta u_j^{(1,2)}(x,t)+\nabla u_j^{(1)}\cdot \nabla u_j^{(2)}\medskip\\
			\hspace*{3cm}= F^{(1)}m_j^{(1,2)}+F^{(2)}(x)m_j^{(1)}m_j^{(2)} & \text{ in } \Omega\times(0,T),\medskip\\
			\p_t m_j^{(1,2)}-\Delta m_j^{(1,2)}-\Delta u_j^{(1,2)}= {\rm div} (m_j^{(1)}\nabla u_j^{(2)})+{\rm div}(m_j^{(2)}\nabla u_j^{(1)}) ,&\text{ in } \Omega\times (0,T) \medskip\\
		    u_j^{(1,2)}(x,T)=0 & \text{ in } \Omega,\medskip\\
			m_j^{(1,2)}(x,0)=0 & \text{ in } \Omega.\\
		\end{cases}  	
	\end{equation}
	By the proof in Step~I, we have $ (u_1^{(1)},m_1^{(1)})=( u_2^{(1)},m_2^{(1)})$.

	Define $\overline{u}^{(1,2)}=u_1^{(1,2)}-u_2^{(1,2)} $ and $\overline{m}^{(1,2)}=m_1^{(1,2)}-m_2^{(1,2)} $. Since  
	$\mathcal{N}_{F_1}=\mathcal{N}_{F_2}$,  we have
	\begin{equation}
		\begin{cases}
			-\overline{u}^{(1,2)}_t-\Delta \overline{u}^{(1,2)}- F_1\overline{m}^{(1,2)}=(F^{(2)}_1-F^{(2)}_2)m_1^{(1)}m_2^{(1)}   & \text{ in } Q\\
			\overline{m}^{(1,2)}_t-\Delta \overline{m}^{(1,2)}-\Delta \overline{u}^{(1,2)}=0  & \text{ in } Q, \medskip\\
		\p_{\nu}\overline{u}^{(1,2)}(x,t)=	\p_{\nu}\overline{m}^{(1,2)}(x,t)=0    & \text{ on } \Sigma, \medskip\\
			\overline{u}^{(1,2)}(x,T)=0	   &\text{ in } \Omega, \medskip\\
			\overline{u}^{(1,2)}(x,0)=0    & \text{ in } \Omega, \medskip\\
			\overline{m}^{(1,2)}(x,0)=\overline{m}^{(1,2)}(x,T)=0& \text{ in } \Omega.\\
		\end{cases}
	\end{equation}
		Let 
	$(v,\rho)$ be a solution to the following system
	\begin{equation}
		\begin{cases}
			v_t-\Delta v= F^{(1)}_1\rho & \text{ in } Q, \medskip\\
			-\rho_t-\Delta \rho-\Delta v=0  & \text{ in } Q, \medskip\\
			\p_{\nu}v(x,t)=\p_{\nu}\rho(x,t)=0      & \text{ on } \Sigma, \medskip\\
		\end{cases}
	\end{equation}
	By Lemma $\ref{key}$, we have
	\begin{equation}
		\int_Q \quad( F^{(2)}_1-F^{(2)}_2)m^{(1)} _1m_2^{(1)} \rho \ dxdt =0.
	\end{equation}
Now we may choose $\rho=e^{-\lambda_it}\overline{m}_i(x)$ and 
$$m_1^{(1)}(x,t)=m_2^{(1)}(x,t)=-\lambda_je^{-\lambda_j t}\overline{m}_j(x)+D_j e^{\lambda_j t}\overline{m}_j(x),$$
for a fixed  $j\in\mathbb{N}.$
Next, by a similar argument in the proof of Step~I, we can derive that 
\begin{equation}
	\int_{\Omega} \quad( F^{(2)}_1-F^{(2)}_2)\overline{m}_j^2\overline{m}_i(x) \ dxdt =0,
\end{equation}
for all $i\in\mathbb{N}$.

Since $\overline{m}_i(x) $ form a complete $L^2$-basis, we have $( F^{(2)}_1-F^{(2)}_2)\overline{m}_j^2=0.$ Notice that the zero set of Neumann eigen-function must be measure zero, we have

$$F^{(2)}_1(x)=F^{(2)}_2(x).$$
	\bigskip
	
	\noindent{Step~III.}~Finally, by mathematical induction and repeating similar arguments as those in the proofs of Steps I and II, one can show that
	$$F^{(k)}_1-F^{(k)}_2=0 ,$$
	for all $k\in\mathbb{N}$. Hence, $F_1(x,m)=F_2(x,m).$
	
	The proof is complete. 
	
\end{proof}

\section*{Acknowledgements}

The work of was supported by the Hong Kong RGC General Research Funds (projects 11301122, 12301420 and 11300821),  the NSFC/RGC Joint Research Fund (project N\_CityU 101/21), and the France-Hong Kong ANR/RGC Joint Research Grant, A-CityU203/19.

\end{document}